\documentclass[smallextended, nospthms]{svjour3}
\usepackage{newclude}
\usepackage[english]{babel}
\usepackage[T1]{fontenc}
\usepackage[utf8]{inputenc}
\usepackage{fullpage}
\usepackage{amsmath,amsfonts,amsthm,amssymb}
\usepackage{enumitem}
\usepackage{tikz}
\usepackage{diagbox}
\usepackage{standalone}
\usepackage{authblk}

\usepackage{textcomp}
\usetikzlibrary{positioning}
\usetikzlibrary{fit}
\usetikzlibrary{arrows}
\usetikzlibrary{arrows.meta}
\usetikzlibrary{calc}
\usetikzlibrary{decorations.pathreplacing}

\usepackage{xskak}
\usepackage{float}

\makeatletter
\newtheorem*{rep@theorem}{\rep@title}
\newcommand{\newreptheorem}[2]{%
	\newenvironment{rep#1}[1]{%
		\def\rep@title{#2 \ref{##1}}%
		\begin{rep@theorem}}%
		{\end{rep@theorem}}}
\makeatother
\newreptheorem{theorem}{Theorem}
\newreptheorem{obs}{Observation}
\newtheorem{theorem}{Theorem}
\newtheorem{lemma}[theorem]{Lemma}
\newtheorem{definition}{Definition}
\newtheorem{corollary}[theorem]{Corollary}

\newtheorem{proposition}[theorem]{Proposition}
\newtheorem{obs}[theorem]{Observation}
\newtheorem{remark}[theorem]{Remark}

\newcommand{\grundy}{\mathcal{G}}
\newcommand{\outcomeP}{\mathcal{P}}
\newcommand{\outcomeN}{\mathcal{N}}
\newcommand{\nimsum}{\oplus}
\newcommand{\arck}{\textsc{Arc-Kayles}}
\newcommand{\lejeu}{\textsc{WAK}}
\newcommand{\nodek}{\textsc{Node-Kayles}}
\DeclareMathOperator{\hasLoop}{loop}
\DeclareMathOperator{\pos}{pos}

\tikzset{vertex/.style={draw, circle, minimum size = 15pt, inner sep = 0pt}}
\tikzset{move arrow/.style={line width=1mm, -{Stealth[angle=60:7pt, inset=0pt]}}}

\newcommand{\selfloop}[1]{\draw[looseness = 8] (#1) edge[in = 60, out=120] (#1) ;}
\newcommand{\dashedloop}[1]{\draw[looseness = 8, dashed] (#1) edge[in = 60, out=120] (#1) ;}

\newcommand{\soloop}[1]{\tikz[baseline=-4]{\useasboundingbox (-0.3,-0.3) rectangle (0.3,0.9) ; \node[vertex] (2) at (0,0) {$#1$}; \selfloop{2} ;}}
\newcommand{\twoloops}[2]{\tikz[baseline=-4]{ \useasboundingbox (-0.3,-0.3) rectangle (1.3,0.9) ; \node[vertex] (2) at (0,0) {$#1$}; \selfloop{2} ;\node[vertex] (1) at (1,0) {$#2$}; \selfloop{1} ;\draw(1)--(2);}}
\DeclareMathOperator{\opt}{opt}
\DeclareMathOperator{\mex}{mex}

\everymath{\displaystyle}

\title{A generalization of \arck}
\author{Antoine Dailly \and Valentin Gledel \and Marc Heinrich}
\institute{Antoine Dailly \and Valentin Gledel \and Marc Heinrich  \at Univ Lyon, Université Lyon 1, LIRIS UMR CNRS 5205, F-69621, Lyon, France.}

\begin{document}
	
	\maketitle
	
	\begin{abstract}
		The game \arck\ is played on an undirected graph with two players taking turns deleting an edge and its endpoints from the graph. We study a generalization of this game, \textsc{Weighted Arc Kayles} (\lejeu\ for short), played on graphs with counters on the vertices. The two players alternate choosing an edge and removing one counter on both endpoints. An edge can no longer be selected if any of its endpoints has no counter left. The last player to play a move wins. We give a winning strategy for \lejeu\ on trees of depth $2$. Moreover, we show that the Grundy values of \lejeu\ and \arck\ are unbounded. We also prove a periodicity result on the outcome of \lejeu\ when the number of counters is fixed for all the vertices but one. Finally, we show links between this game and a variation of the non-attacking queens game on a chessboard.
		
		\keywords{Combinatorial Games \and Arc-Kayles \and Graphs}
	\end{abstract}
	
	\section*{Acknowledgements}
	This work has been supported by the ANR-14-CE25-0006 project of the French National Research Agency.
	
	\noindent The authors would like to thank Nicolas Bousquet for his help.

	
	\section{Introduction}

	Combinatorial games are finite two-player games without chance nor hidden information. Combinatorial Game Theory (see \cite{siegel} for a survey) was developed to analyze games when the winner is determined by the last move.
	For these games, one of the player has a winning strategy, \emph{i.e.}, one player is guaranteed to win the game, whatever the other player does. It raises three natural questions: which player has a winning strategy? What is this strategy? Can we compute it efficiently?

	In 1978, Schaefer \cite{schaefer} introduced several combinatorial games on graphs. Among them is the game \arck. In this game, players take turns deleting an edge and its endpoints from the graph, until no edge remain. The winner is the player making the last move. Another way to describe this game is the following: the two players select edges in order to build a maximal matching. The first player's goal is to create a matching of odd size, while the second player tries to make it of even size.
	
	Schaefer introduced \arck\ as a variant of \nodek, which is a game where the players alternate selecting a vertex and deleting it and all its neighbours from a graph. 
	He proved that \nodek\ is PSPACE-complete.
	This game has then been studied on specific graph classes:
	it has been proven that deciding its outcome is polynomial when playing on graphs with bounded asteroidal number, cocomparability graphs, cographs~\cite{bodKayles} and bounded degree stars~\cite{fleischer}. A derived game called \textsc{Grim} has also been studied in~\cite{grim}. In this variant, the neighbours of the selected vertex are deleted from the graph if and only if they become isolated.
	
	It is still an open question whether the problem of deciding which player has a winning strategy for \arck\ is PSPACE-complete or not, and very few results, either general or on specific graph classes, are known. Played on a path, it is equivalent to the octal game $0.07$, also called \textsc{Dawson's Kayles}, solved in~\cite{guysmith}. More recently, some results have been found for specific classes of graphs: cycles, wheels and subdivided stars with three paths~\cite{huggan}. It was also shown in~\cite{valia} that the problem is FPT when parameterized by the number of rounds, meaning it can be solved in time $O(f(k) + \text{poly}(n))$ where $k$ is the number of rounds of the game, $n$ is the number of vertices and $f$ is some computable function.

	We study a generalization of \arck, called \textsc{Weighted-Arc-Kayles} (\lejeu\ for short). The game \lejeu\ is played on an undirected graph with counters on the vertices. The players take turns selecting an edge, and removing one counter from each of its endpoints. If an edge has an endpoint with no counter left, then it cannot be selected anymore. The game ends when no edge can be selected anymore.
	When there is only one counter on each vertex, this game is exactly \arck.
	This game was first proposed by Huggan in~\cite{H15}, along with several others, as a possible extension of \arck.
	Our study of \lejeu\ is also motivated by a variation of the non-attacking queens game~\cite{noon2006non}. Consider the game where players alternately place non-attacking rooks on a not necessarily square chessboard. We show that this game can be represented as an instance of~\lejeu. 
	
	In order to simplify the study of some graphs, we allow the vertices to have loops.
	In this paper, we prove that one can decide in polynomial time which player has a winning strategy on loopless trees of depth at most 2 (we consider that a tree reduced to a single vertex has depth~0). This directly solves a particular case of the non-attacking rooks game.
	\begin{theorem}
		\label{thm:mainResult}
		There is a polynomial time algorithm computing the outcome of \lejeu\ on any loopless tree of depth at most $2$.
	\end{theorem}
	The outcome of \lejeu\ on a loopless tree of depth at most $2$ is determined by the parity of the sum of the numbers of counters of some sets of vertices, and by inequalities between them. This is not the case with $C_3$, the cycle on three vertices, which suggests that it might be harder to characterize outcomes for graphs with induced cycles, or at least non-bipartite graphs.
	
	Grundy values are a tool used in Combinatorial Game Theory to refine the question of which of the two players wins (a more formal definition is given in Section~\ref{sec:defnot}). The Grundy values of \arck\ (and by extension of \nodek) were conjectured to be unbounded in~\cite{huggan}. We give a positive answer to this conjecture as a corollary of the following:
	\begin{theorem}
		\label{prop:grundyValuesUnbounded}
		The Grundy values for the game \lejeu~are unbounded.
	\end{theorem}
	
	The paper is organized as follows: in Section~\ref{sec:defnot}, we give basic definitions and formally define~\lejeu. Section~\ref{sec:tours} shows the links between~\lejeu\ and the game of placing non-attacking rooks on a chessboard. In Section~\ref{sec:first}, we define the core concept of \emph{canonical graphs}. This notion is used to prove a relation between \lejeu\ and \arck. It also simplifies the study of graphs in Section~\ref{sec:trees}, where we characterize which player has a winning strategy for~\lejeu\ on loopless trees of depth at most~2. This characterization only depends on the parity of the weights, and some inequalities between them. Next, we present in Section~\ref{sec:perio} a periodicity result on the outcomes of \lejeu\ positions when the number of counters is fixed for all but one vertex. Finally, we prove in Section~\ref{sec:unbound} that the Grundy values of \lejeu\ and \arck\ are unbounded.

	\section{Definitions and notations}
	\label{sec:defnot}
	
	\subsection{Combinatorial Game Theory}
	
	We will give basic definitions of Combinatorial Game Theory that will be used in the paper. For more details, the interested reader can refer to~\cite{lip,winningways,siegel}.
	\emph{Combinatorial games} \cite{winningways} are two-player games where: 
	\begin{itemize}[noitemsep,topsep=0pt]
		\item the players play alternately;
		\item there is no chance, nor hidden information;
		\item the game is finite; 
		\item the winner is determined by the last move alone.
	\end{itemize}
	
	
	In this paper, the games are \emph{impartial}, i.e., both players have exactly the same set of available moves on any position. The only difference between the two players is who plays the first move.
	Every position~$G$ of a combinatorial game can be viewed as a combinatorial game with~$G$ as the initial position. By abuse of notation, we will often consider positions as games.

	From a given position $G$, the positions that can be reached from~$G$ by playing a move are called the \emph{options} of~$G$. The set of the options of~$G$ is denoted~$\opt(G)$.		
	A position of an impartial game can have exactly two \emph{outcomes}: either the first player has a winning strategy and it is called an \emph{$\outcomeN$-position} (for "$\outcomeN$ext player win"), or the second player has a winning strategy and it is called a \emph{$\outcomeP$-position} (for "$\outcomeP$revious player win").
	The outcome of a position $G$ can be computed recursively from the outcome of its options using the following characterization:
	
	\begin{proposition}[\cite{lip}]
		\label{prop:outcomes}
		Let $G$ be a position of an impartial game in normal play.
		\begin{itemize}[noitemsep, topsep=0pt]
			\item If $opt(G)=\emptyset$, then $G$ is a $\outcomeP$-position;
			\item If there exists a position $G' \in opt(G)$ such that $G'$ is a $\outcomeP$-position, then $G$ is an $\outcomeN$-position and a winning move is to play from $G$ to $G'$;
			\item If every option of $G$ is an $\outcomeN$-position, then $G$ is a $\outcomeP$-position.
		\end{itemize}
	\end{proposition}
	
	Although not classical, we define a relation between games based on their outcomes: if two games $G_1$ and $G_2$ have the same outcome, they are \emph{outcome-equivalent}, and we write $G_1 \sim G_2$.
	
	Given two games $G_1$ and $G_2$, we define their \emph{disjoint sum}, denoted $G_1 + G_2$, as the game where, at their turn, the players play a legal move on either $G_1$ or $G_2$ until both games are finished. The player making the last move wins.
	If $G_1$ is a $\outcomeP$-position, then $G_1 + G_2$ has the same outcome as $G_2$. Indeed, the player with a winning strategy on $G_2$ can apply this strategy and reply to any move made by his opponent on $G_1$ using the second player's winning strategy on $G_1$.
	In order to study the outcome of $G_1 + G_2$ in the case where both $G_1$ and $G_2$ are $\outcomeN$-positions, we refine the outcome-equivalence by a relation called the \emph{Grundy-equivalence}. Two games $G_1$ and $G_2$ are \emph{Grundy-equivalent}, denoted by $G_1 \equiv G_2$, if and only if for any game $G$, $G_1+G$ and $G_2+G$ have the same outcome, i.e., $G_1+G \sim G_2+G$.
	In particular, if $G_1 \equiv G_2$, then $G_1+G_2$ is a $\outcomeP$-position.
	
	We can attribute a value to a game according to its Grundy equivalence class, called the \emph{Grundy value}. The Grundy value of a game position $G$, denoted $\grundy(G)$ can be computed using the Grundy values of its options thanks to the following formula:
	$$
	\grundy(G)=mex(\grundy(G') | G' \in opt(G))
	$$
	where, given a finite set of nonnegative integers $S$, $mex(S)$ is the smallest nonnegative integer not in $S$.
	In particular, a position $G$ is a $\outcomeP$-position if and only if $\grundy(G)=0$, which is consistent with Proposition~\ref{prop:outcomes}.
	
	One of the most fundamental results in Combinatorial Game Theory is the Sprague-Grundy Theorem, which gives the Grundy value of the disjoint sum of two impartial games:
	
	\begin{theorem}[Sprague-Grundy Theorem \cite{Spra36}] \label{thm:sprgru}
		Given two impartial games $G_1$ and $G_2$, we have 
		$$\grundy(G_1+G_2)=\grundy(G_1)\nimsum\grundy(G_2),$$
		where $\nimsum$, called the Nim-sum, is the bitwise XOR. 
	\end{theorem}

	\subsection{Definition of \lejeu~and notations}
	
	In the whole paper we consider weighted graphs $G=(V,E,\omega)$, where $V$ is the set of vertices, $E$ is the set of edges, and $\omega: V \rightarrow \mathbb{N}^{*}$ gives the number of counters on each vertex.
	Graphs are undirected. There may be loops, in which case an edge looping on a vertex $u$ is denoted by $(u,u)$. For every vertex $u$, $\hasLoop(u)$ is a boolean with value \texttt{True} if and only if the edge $(u,u)\in E$, in which case we say that a loop is \emph{attached to} $u$.
	
	
	\begin{figure}[!ht]
		\centering
		\begin{tikzpicture}
		\node[scale=0.8] (tamere) at (-4,-3) {Play on the leftmost edge.};
		\node[scale=0.8] (tamerelareloue) at (0,-3) {Play on the loop.};
		\node[scale=0.8] (tamerelamacron) at (4,-3) {Play on the rightmost edge.};
		\node[scale=0.8] (tamerelafraise) at (4,-3.25) {This edge cannot be selected anymore.};
		\node[draw, inner sep = 0pt] (orig) at (0,0) {
			\begin{tikzpicture}
			\clip (-0.4,-0.3) rectangle (2.4,0.9) ;
			\node[vertex] (1) at (0,0) {2};
			\node[vertex] (2) at (1,0) {5};
			\node[vertex] (3) at (2,0) {1};
			\draw (1) to (2);
			\draw (2) to (3);
			\selfloop{2} ;
			\end{tikzpicture}
		};
		\node (pos1) at (-4,-2) {
			\begin{tikzpicture}
			\clip (-0.4,-0.3) rectangle (2.4,0.9) ;
			\node[vertex] (1) at (0,0) {1};
			\node[vertex] (2) at (1,0) {4};
			\node[vertex] (3) at (2,0) {1};
			\draw[] (1) to (2);
			\draw (2) to (3);
			\selfloop{2} ;
			\end{tikzpicture}
		};
		\node (pos2) at (0,-2) {
			\begin{tikzpicture}
			\clip (-0.4,-0.3) rectangle (2.4,0.9) ;
			\node[vertex] (1) at (0,0) {2};
			\node[vertex] (2) at (1,0) {4};
			\node[vertex] (3) at (2,0) {1};
			\draw (1) to (2);
			\draw (2) to (3);
			\draw[looseness = 8] (2) edge[in = 60, out=120] (2) ;
			\end{tikzpicture}
		};
		\node (pos3) at (4,-2) {
			\begin{tikzpicture}
			\clip (-0.4,-0.3) rectangle (2.4,0.9) ;
			\node[vertex] (1) at (0,0) {2};
			\node[vertex] (2) at (1,0) {4};
			\node[vertex] (3) at (2,0) {0};
			\draw (1) to (2);
			\draw (2) to (3);
			\selfloop{2} ;
			\end{tikzpicture}
		};
		\clip (orig.south) + (-4,0) rectangle (4,-2) ;
		\draw[->, line width = 2pt, shorten <=-2pt] (orig) -- (pos1);
		\draw[->, line width = 2pt] (orig) -- (pos2);
		\draw[->, line width = 2pt, shorten <=-2pt] (orig) -- (pos3);
	\end{tikzpicture}
	\caption{\label{fig:graphGame} Example of possible moves for Weighted Arc Kayles on a given position.}
\end{figure}
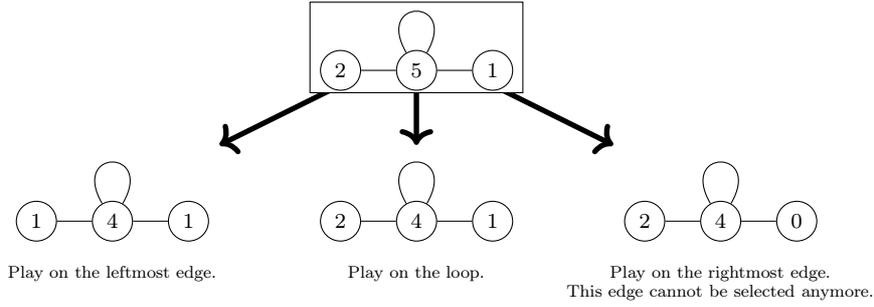

At each turn, the current player selects an edge and removes one counter from both of its endpoints (or its unique endpoint, if the edge is a loop). For any vertex $u$ such that $\omega(u)=0$, the edges $(u,v)$ cannot be selected anymore.
The game then continues until no edge can be selected anymore.
Figure~\ref{fig:graphGame} shows an example of the moves available from a given position. 
For simplicity and if confusion is not possible, vertices on caption of figures will be named after their number of counters.

We consider that, when selecting a loop, one counter is removed from its unique endpoint. This convention is defined since we introduce loops to simplify the study of several graphs. The other possible convention, which would remove two counters from the endpoint, could also be considered.

If the number of counters of one vertex reaches zero, then the edges adjacent to this vertex cannot be played anymore and the vertex can be removed from the graph. Note that if the graph $G$ is not connected, then \lejeu\ played on $G$ is equivalent to the disjoint sum of the connected components of $G$. Consequently, if we can compute the Grundy values of each of the connected components, we can use Theorem~\ref{thm:sprgru} and get the Grundy value of $G$. 

\begin{obs}
	\label{prop:playingOnANonConnectedGraph}
	Let $G$ be a non-connected graph, and $G_1,\ldots,G_k$ its connected components.
	We have $$\grundy(G)=\grundy\big(\sum_{i=1}^{k} G_i\big)=\bigoplus_{i=1}^{k}\grundy(G_i).$$
\end{obs}

\section{Relation with non-attacking rooks on a chessboard}
\label{sec:tours}

Inspired by the non-attacking queens game~\cite{noon2006non}, we introduce the non-attacking rooks game. It is played on an $n\times m$ chessboard $\mathcal C$. There is a subset $\mathcal{H}$ of the squares of the chessboard whose elements are named \textit{holes}.
At each turn, the current player places a rook on a square of the chessboard that is not a hole in such a way that it does not attack any of the already played rooks.
The rooks cannot 'jump over' holes. In other words, there can be two rooks on the same row provided there is a hole between them. The first player unable to play loses.

When the chessboard has no holes, the game is fully characterized by the parity of the minimum dimension of the grid. Indeed, at each turn a row and a column are deleted from the chessboard, and the game ends when there is no more row or column to play on. An example of such a game can be seen on Figure~\ref{fig:tours}.

\setchessboard{
	maxfield = h6,
	setpieces={Re5,Ra3},
	showmover=false
}

\begin{figure}[!ht]
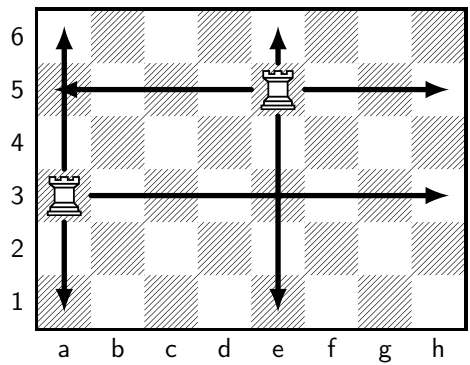

	\centering
	\chessboard[
	pgfstyle=straightmove,
	shorten = 0.4em,
	shortenend = -0.2em,
	backmove = {e5-e1,e5-e6,e5-a5,e5-h5,a3-a1,a3-a6,a3-h3}
	]
	\caption{On a $6 \times 8$ chessboard without holes the game always end after six moves and the second player wins.}
	\label{fig:tours}
\end{figure}

We prove that the non-attacking rooks game can be viewed a special case of \lejeu, played on a certain graph.

\begin{proposition}
	Every position of the non-attacking rooks game has an equivalent position in \lejeu. 
\end{proposition}

\begin{proof}
	
	We define a \textit{vertical rectangle cover} of the chessboard as a set $\mathcal R_V$ of rectangles on the chessboard, such that no two rectangles intersect, no rectangle contains holes, the union of the rectangles contains all of the squares of the chessboard that are not holes, and all the squares directly above or below each rectangle are either holes or outside the board.
	
	
	We can similarly define a \textit{horizontal rectangle cover} where the squares at the left and right border are either holes or outside the chessboard. An example of such covers can be seen on Figure~\ref{fig:rectCover}.
	We can always find such rectangle covers, for example by taking pieces of rows or columns between the holes.
	
	Consider a position of the non-attacking rooks game, and let $\mathcal{R}_V$ and $\mathcal{R}_H$ be respectively a vertical and horizontal rectangle cover of the board. Let $G=(V,E, \omega)$ be the weighted graph built as follows: 
	\begin{itemize}
		\item for each vertical rectangle $V_i$, add a vertex $v_i$ with weight the number of columns of $V_i$ minus the number of rooks already present in $V_i$, 
		\item for each horizontal rectangle $H_j$, add a vertex $h_j$ with weight the number of rows of $H_j$ minus the number of rooks in $H_j$,
		\item if two rectangles $V_i$ and $H_j$ intersect, then add an edge between $v_i$ and $h_j$.
	\end{itemize}

	
	
	Each time a player places a rook on the chessboard, it is inside exactly one vertical rectangle $V_i$ and one horizontal rectangle $H_j$. Since it is forbidden to attack this rook, it is equivalent to remove a column from $V_i$ and a line from $H_j$. So the graph obtained from this new position is the one in which the weights of $v_i$ and $h_j$ are decreased by one. The position of the rook in the intersection of $V_i$ and $H_j$ does not matter.
	
	Moreover, in a graph $G$ associated with a position of the non-attacking rooks game, for each move on an edge $(v_i,h_j)$, the weights of the vertices are positives so the corresponding rectangles have a positive number of line or columns. Moreover, since the edge $(v_i,h_j)$ exists, the rectangles intersect. By construction, the intersection of $V_i$ and $H_j$ is a rectangle with the same number of columns as $V_i$, and the same number of rows as $H_j$. Additionally, there are exactly $\omega(v_i)$ free columns, i.e., not occupied by previously played rooks, and $\omega(h_j)$ free rows. Consequently, it is possible to place a rook in the intersection of $V_i$ and $H_j$.
	
	Denote by $f$ the construction above that transforms a position $G$ for the rook placement game into a position $f(G)$ for \lejeu. For each move from $G$ to $G'$ for the non-attacking rooks game, there is an equivalent move from $f(G)$ to $f(G')$. Conversely, for every move from $f(G)$ to some position $G_1$ of \lejeu, there is an equivalent move from $G$ to some position $G' \in f^{-1}(G_1)$. Consequently, the two games are equivalent.

\end{proof}

\begin{figure}[t]
	\centering
	\begin{tikzpicture}
	\node (v) at (0,0) {
		\scalebox{0.5}{
			\begin{tikzpicture}[scale = 0.72]
			\draw[very thick] (0,0) rectangle (12,10);
			\draw[fill = black] (3,5) rectangle (6,7);
			\draw[fill = black] (4,4) rectangle (6,5);
			\draw[fill = black] (4,3) rectangle (8,4);
			\draw[fill = black] (10,4) rectangle (12,6);
			
			\draw (0.1,0.1) rectangle (2.9,9.9) node[midway]{$V_1$};
			\draw (3.1,0.1) rectangle (3.9,4.9) node[midway]{$V_2$};
			\draw (3.1,7.1) rectangle (5.9,9.9) node[midway]{$V_3$};
			\draw (4.1,0.1) rectangle (7.9,2.9) node[midway]{$V_4$};
			\draw (6.1,4.1) rectangle (7.9,9.9) node[midway]{$V_5$};
			\draw (8.1,0.1) rectangle (9.9,9.9) node[midway]{$V_6$};
			\draw (10.1,0.1) rectangle (11.9,3.9) node[midway]{$V_7$};
			\draw (10.1,6.1) rectangle (11.9,9.9) node[midway]{$V_8$};
			
			\end{tikzpicture}	
		}
	};
	
	\node (h) at (4.8,0) {
		\scalebox{0.5}{
			\begin{tikzpicture}[scale = 0.72]
			\draw[very thick] (0,0) rectangle (12,10);
			\draw[fill = black] (3,5) rectangle (6,7);
			\draw[fill = black] (4,4) rectangle (6,5);
			\draw[fill = black] (4,3) rectangle (8,4);
			\draw[fill = black] (10,4) rectangle (12,6);
			
			\draw (0.1,7.1) rectangle (11.9,9.9) node[midway]{$H_1$};
			\draw (0.1,5.1) rectangle (2.9,6.9) node[midway]{$H_2$};
			\draw (6.1,6.1) rectangle (11.9,6.9) node[midway]{$H_3$};
			\draw (6.1,4.1) rectangle (9.9,5.9) node[midway]{$H_4$};
			\draw (0.1,3.1) rectangle (3.9,4.9) node[midway]{$H_5$};
			\draw (8.1,3.1) rectangle (11.9,3.9) node[midway]{$H_6$};
			\draw (0.1,0.1) rectangle (11.9,2.9) node[midway]{$H_7$};
			
			\end{tikzpicture}	
		}
	};
	
	\node (g) at (10.5,0) {
		\scalebox{0.6}{
			\begin{tikzpicture}[scale = 1]
			\node[vertex] (h2) at (3,3.2) {$H_2$};
			\node[vertex] (h1) at (1.5,0) {$H_1$};
			\node[vertex] (h3) at (3,0) {$H_3$};
			\node[vertex] (h4) at (4.5,0) {$H_4$};
			\node[vertex] (h5) at (5.5,0) {$H_5$};
			\node[vertex] (h6) at (6.8,0) {$H_6$};
			\node[vertex] (h7) at (9,0) {$H_7$};
			
			\node[vertex] (v1) at (4.5,4) {$V_1$};
			\node[vertex] (v2) at (6.5,1.5) {$V_2$};
			\node[vertex] (v3) at (1,-2) {$V_3$};
			\node[vertex] (v4) at (9.5,-2) {$V_4$};
			\node[vertex] (v5) at (3.5,1.7) {$V_5$};
			\node[vertex] (v6) at (4.5,-4) {$V_6$};
			\node[vertex] (v7) at (7.5,-0.7) {$V_7$};
			\node[vertex] (v8) at (2.5,-0.7) {$V_8$};
			
			\draw (h1)--(v1);
			\draw (h1)--(v3);
			\draw (h1)--(v5);
			\draw (h1)--(v6);
			\draw (h1)--(v8);
			
			\draw (h2)--(v1);
			
			\draw (h3)--(v5);
			\draw (h3)--(v6);
			\draw (h3)--(v8);
			
			\draw (h4)--(v5);
			\draw (h4)--(v6);
			
			\draw (h5)--(v1);
			\draw (h5)--(v2);
			
			\draw (h6)--(v6);
			\draw (h6)--(v7);
			
			\draw (h7)--(v1);
			\draw (h7)--(v2);
			\draw (h7)--(v4);
			\draw (h7)--(v6);
			\draw (h7)--(v7);
			
			\end{tikzpicture}	
		}
	};
	
\end{tikzpicture}
\caption{\label{fig:rectCover} Vertical and horizontal rectangles cover and the graph associated.}

\end{figure}
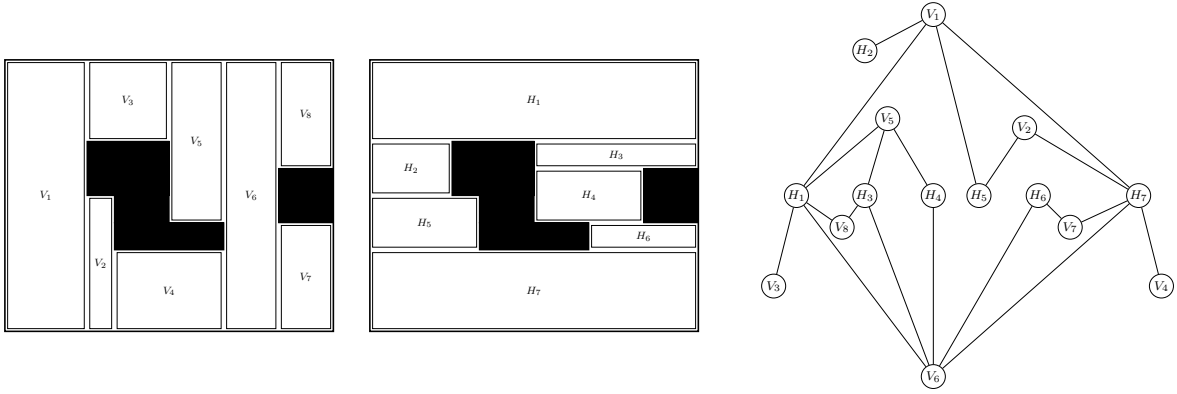

\begin{figure}[!ht]
\centering
\scalebox{0.7}{
	
	\begin{tikzpicture}

	\node (b) at (8.5,0) {
		\scalebox{0.9}{
			\begin{tikzpicture}[scale = 0.72]
			\draw (0,0) rectangle (10,7);
			\draw[fill = black] (3,4) rectangle (7,7);
			\draw[<->, thick] (0.1,7.2) -- (2.75,7.2) node[midway,above]{$n_1$};
			\draw[<->, thick] (7.25,7.2) -- (9.9,7.2) node[midway,above]{$n_2$};
			\draw[<->, thick] (2.8,4) -- (2.8,6.9) node[midway,left]{$h_1$};
			\draw[<->, thick] (7.2,4) -- (7.2,6.9) node[midway,right]{$h_2$};
			\draw[<->, thick] (0.2,0.1) -- (0.2,3.8) node[midway,right]{$m_1$};
			\draw[<->, thick] (3,3.8) -- (7,3.8) node[midway,below]{$w_1$};
			
			\draw[dotted] (2.8,0) -- (2.8,4) ;
			\draw[dotted] (7.2,0) -- (7.2,4) ;
			\draw[dotted] (0,3.8) -- (3,3.8) ;
			\draw[dotted] (7,3.8) -- (10,3.8) ;
			
			\end{tikzpicture}
		}
	};
	
	\node (d) at (17,0) {
		
		\scalebox{1.35}{
			
			\begin{tikzpicture}   
			
			\node[vertex] (H1) at (0,0) {$h_1$};
			\node[vertex] (N1) at (1,0) {$n_1$};
			\node[vertex] (N2) at (3,0) {$n_2$};
			\node[vertex] (H2) at (4,0) {$h_2$};
			\node[vertex] (M1) at (2,-1) {$m_1$};
			\node[vertex] (W1) at (2,-2) {$w_1$};        
			
			\draw (H1) -- (N1);
			\draw (H2) -- (N2);
			\draw (W1) -- (M1);
			\draw (M1) -- (N1);
			\draw (M1) -- (N2);      
			
			\end{tikzpicture}
		}        
	};
	
	\draw (12.9,0) node[scale=3] {$\Leftrightarrow$};
\end{tikzpicture}

}
\caption{Reduction when there is a hole by the edge of the chessboard.}
\label{fig:tours_bord}
\end{figure}
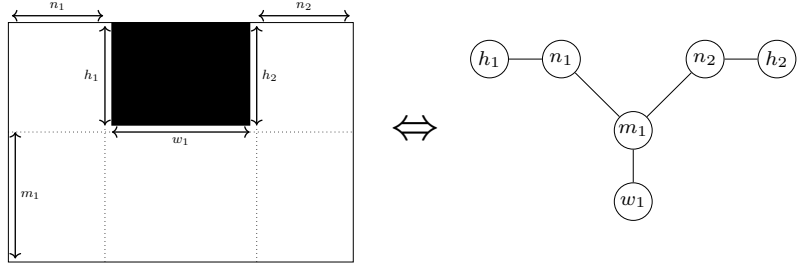

From the proof of the proposition above, we can see that the exact square on which a rook is placed is not important. Indeed, what really matters is in which area it is placed. For example, in the case of a rectangular hole by the edge of the chessboard, as in Figure~\ref{fig:tours_bord}, there are $5$ different rectangular areas: top left, top right, bottom left, bottom right, and bottom center. Consequently, there are only $5$ types of moves in this case, and the only thing that a player has to choose is in which area he will play. In this case, the equivalent \lejeu\ position is a tree of depth 2, which is solved by Theorem~\ref{thm:mainResult}. The more general case with a hole in the middle of the board seems more complicated.


The rest of paper will now be dedicated to study properties of \lejeu .

\section{Canonical Graphs}
\label{sec:first}

The notion of canonical graphs comes from the observation that some vertices have little influence on the outcome of the game. The following definition identifies these vertices. Given a graph $G(V,E,\omega)$ and a vertex $u\in V$, we define the \emph{neighbourhood} of $u$, denoted by $N(u)$, as the set $\{v \in V | (u,v) \in E\}$.

\begin{definition}
Let $G=(V,E,\omega)$ be a weighted graph. We define the following:
\begin{itemize}[noitemsep, topsep=0pt]
\item A vertex $u \in V$ is \emph{useless} if $u$ has no loop attached to it, and all the neighbors of $u$ have a loop attached to them.
\item A vertex $u \in V$ is \emph{heavy} if $u$ has no loop attached to it, and $\omega(u) \geq \sum_{v \in N(u)} \omega(v)$.
\item Two non-adjacent vertices $u$ and $v$ are \emph{false twins} if $N(u) = N(v)$, and $\hasLoop(u) = \hasLoop(v)$.
\end{itemize}
\end{definition}

The idea is that whenever there is a useless vertex, a heavy vertex, or two false twins, the graph can be simplified. This simplification is described in the definition below. It is illustrated on Figure~\ref{fig:reduction}.


\begin{definition}
Let $G=(V,E,\omega)$ be a weighted graph. A \emph{reduction} of $G$ is a graph $G'$ obtained by applying any arbitrary sequence of the following steps:
\begin{itemize}[noitemsep, topsep=0pt]
\item Deleting a useless vertex $u$;
\item Deleting a heavy vertex $u$ and attaching a loop to each of its neighbors;
\item Merging two false twins $v_1$ and $v_2$ into a single vertex $v$ with weight $\omega(v_1) + \omega(v_2)$.
\end{itemize}
\end{definition}

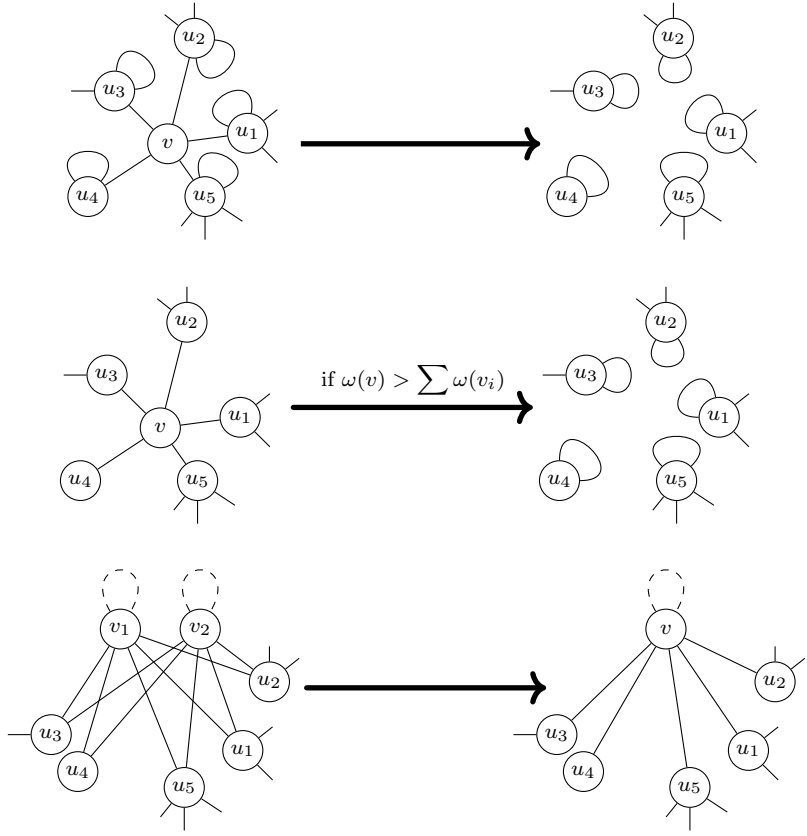
\begin{figure}[!ht]
\centering
\begin{tikzpicture}
\node[] (b) at (0,0) {
\begin{tikzpicture}[scale=0.7]

\node[vertex] (v) at (0,0) {$v$} ;

\node[vertex] (u1) at (1.5,0.2) {$u_1$} ;
\node[vertex] (u2) at (0.5,2) {$u_2$} ;
\node[vertex] (u3) at (-1,1) {$u_3$} ;
\node[vertex] (u4) at (-1.5,-1) {$u_4$} ;
\node[vertex] (u5) at (0.7,-1) {$u_5$} ;

\foreach \i in {1,2,3,4,5} {
	\draw (v) -- (u\i) ;
}

\foreach \j/\pos/\dist in {
	1/below right/0.2cm and 0.2cm, 
	1/above right/0.1cm and 0.2cm, 
	2/above/0.2cm, 
	2/above left/0.1cm and 0.2cm, 
	3/left/0.3cm,
	5/below left/0.2cm and 0.1cm,
	5/below/0.3cm,
	5/below right/0.1cm and 0.3cm
} {
	\node[\pos= \dist of u\j] (temp) {} ;
	\draw (temp) -- (u\j) ;
}

\foreach \i/\in\out in {1/160/80,2/-5/-85,3/85/5,4/50/130,5/25/105} {
	\draw[looseness=5] (u\i) edge[in=\in, out=\out] (u\i) ;
}

\node[vertex] (u1) at (9+1.5,0.2) {$u_1$} ;
\node[vertex] (u2) at (9+0.5,2) {$u_2$} ;
\node[vertex] (u3) at (9-1,1) {$u_3$} ;
\node[vertex] (u4) at (9-1.5,-1) {$u_4$} ;
\node[vertex] (u5) at (9+0.7,-1) {$u_5$} ;

\foreach \j/\pos/\dist in {
	1/below right/0.2cm and 0.2cm, 
	1/above right/0.1cm and 0.2cm, 
	2/above/0.2cm, 
	2/above left/0.1cm and 0.2cm, 
	3/left/0.3cm,
	5/below left/0.2cm and 0.1cm,
	5/below/0.3cm,
	5/below right/0.1cm and 0.3cm
} {
	\node[\pos= \dist of u\j] (temp) {} ;
	\draw (temp) -- (u\j) ;
}

\foreach \i/\in\out in {1/180/100,2/-54/-120,3/35/-35,4/0/90,5/45/135} {
	\draw[looseness=5] (u\i) edge[in=\in, out=\out] (u\i) ;
}

\draw[->, line width = 2pt] (2.5,0) -- (7,0) ;
\end{tikzpicture}
} ;

\node[below = 0cm of b] (c) {

\begin{tikzpicture}[scale =.7]

\node[vertex] (v) at (0,0) {$v$} ;

\node[vertex] (u1) at (1.5,0.2) {$u_1$} ;
\node[vertex] (u2) at (0.5,2) {$u_2$} ;
\node[vertex] (u3) at (-1,1) {$u_3$} ;
\node[vertex] (u4) at (-1.5,-1) {$u_4$} ;
\node[vertex] (u5) at (0.7,-1) {$u_5$} ;

\foreach \i in {1,2,3,4,5} {
	\draw (v) -- (u\i) ;
}

\foreach \j/\pos/\dist in {
	1/below right/0.2cm and 0.2cm, 
	1/above right/0.1cm and 0.2cm, 
	2/above/0.2cm, 
	2/above left/0.1cm and 0.2cm, 
	3/left/0.3cm,
	5/below left/0.2cm and 0.1cm,
	5/below/0.3cm,
	5/below right/0.1cm and 0.3cm
} {
	\node[\pos= \dist of u\j] (temp) {} ;
	\draw (temp) -- (u\j) ;
}

\node[vertex] (u1) at (9+1.5,0.2) {$u_1$} ;
\node[vertex] (u2) at (9+0.5,2) {$u_2$} ;
\node[vertex] (u3) at (9-1,1) {$u_3$} ;
\node[vertex] (u4) at (9-1.5,-1) {$u_4$} ;
\node[vertex] (u5) at (9+0.7,-1) {$u_5$} ;

\foreach \j/\pos/\dist in {
	1/below right/0.2cm and 0.2cm, 
	1/above right/0.1cm and 0.2cm, 
	2/above/0.2cm, 
	2/above left/0.1cm and 0.2cm, 
	3/left/0.3cm,
	5/below left/0.2cm and 0.1cm,
	5/below/0.3cm,
	5/below right/0.1cm and 0.3cm
} {
	\node[\pos= \dist of u\j] (temp) {} ;
	\draw (temp) -- (u\j) ;
}

\foreach \i/\in\out in {1/180/100,2/-54/-120,3/35/-35,4/0/90,5/45/135} {
	\draw[looseness=5] (u\i) edge[in=\in, out=\out] (u\i) ;
}

\draw[->, line width = 2pt, line cap = round] (2.5,0) -- (7,0) node [midway, above] {if $\omega(v) > \sum \omega(v_i)$} ;
\end{tikzpicture}
} ;

\node[below = 0cm of c] (a) {
\begin{tikzpicture}[scale =.7]

\node[vertex] (v1) at (-0.5-0.5,3) {$v_1$} ;
\node[vertex] (v2) at (1-0.5,3) {$v_2$} ;

\node[vertex] (u1) at (1.8-0.5,0.7) {$u_1$} ;
\node[vertex] (u2) at (2.3-0.5,2) {$u_2$} ;
\node[vertex] (u3) at (-1.8-0.5,1) {$u_3$} ;
\node[vertex] (u4) at (-1.3-0.5,0.3) {$u_4$} ;
\node[vertex] (u5) at (0.7-0.5,0) {$u_5$} ;

\foreach \i in {1,2,3,4,5} {
	\draw (v1) -- (u\i) ;
	\draw (v2) -- (u\i) ;
}

\dashedloop{v1} ;
\dashedloop{v2} ;

\foreach \j/\pos/\dist in {
	1/below right/0.2cm and 0.2cm, 
	1/above right/0.1cm and 0.2cm, 
	2/above/0.2cm, 
	2/above right/0.1cm and 0.2cm, 
	3/left/0.3cm,
	5/below left/0.2cm and 0.1cm,
	5/below/0.3cm,
	5/below right/0.1cm and 0.3cm
} {
	\node[\pos= \dist of u\j] (temp) {} ;
	\draw (temp) -- (u\j) ;
}

\node[vertex] (v) at (9.25,3) {$v$} ;

\node[vertex] (u1) at (9+ 1.8,0.7) {$u_1$} ;
\node[vertex] (u2) at (9+ 2.3,2) {$u_2$} ;
\node[vertex] (u3) at (9+ -1.8,1) {$u_3$} ;
\node[vertex] (u4) at (9+ -1.3,0.3) {$u_4$} ;
\node[vertex] (u5) at (9+ 0.7,0) {$u_5$} ;

\foreach \i in {1,2,3,4,5} {
	\draw (v) -- (u\i) ;
}

\foreach \j/\pos/\dist in {
	1/below right/0.2cm and 0.2cm, 
	1/above right/0.1cm and 0.2cm, 
	2/above/0.2cm, 
	2/above right/0.1cm and 0.2cm, 
	3/left/0.3cm,
	5/below left/0.2cm and 0.1cm,
	5/below/0.3cm,
	5/below right/0.1cm and 0.3cm
} {
	\node[\pos= \dist of u\j] (temp) {} ;
	\draw (temp) -- (u\j) ;
}

\dashedloop{v} ;

\draw[->, line width = 2pt, line cap = round] (2.5,1.5) -- (7,1.5) ;
\end{tikzpicture}
} ;
\end{tikzpicture}

\caption{\label{fig:reduction} The three possible reductions for a graph. At the top, $v$ is a useless vertex. In the middle, $v$ is a heavy vertex with $\omega(v) \geq \sum \omega(u_i)$. At the bottom, two false twins are merged, and we have $\omega(v) = \omega(v_1) + \omega(v_2)$, and $\hasLoop(v) = \hasLoop(v_1) = \hasLoop(v_2)$.}
\end{figure}

A graph is \emph{canonical} if it has no useless vertices, no heavy vertices, and no false twins. The following lemma ensures that the reduction of the graph preserves its Grundy value (and thus its outcome).

\begin{proposition}
\label{prop:reduction}
If $G$ is a graph and $G'$ is a reduction of $G$, then $G \equiv G'$.
\end{proposition}
\begin{proof}
We will prove that the none of the three reduction operations changes the Grundy value of the graph. Let $G$ be a weighted graph, and $G'$ be a graph obtained by applying only one reduction operation on~$G$. The proof is by induction on the total sum of all the weights of $G$. If all the weights of $G$ are zero, then clearly their is no move on either $G$ or $G'$, and both games have Grundy value $0$.

Suppose by induction that the property holds for all graphs $G$ with total weight at most $k$. Let $G = (V,E,\omega)$ be a graph with total weight $k+1$, and let $G'$ be a graph obtained from $G$ by applying one step of the reduction rules.

In the following, given an edge $e$ of $G$ (resp. $G'$) such that both its endpoints have positive weights, we denote by $G_e$ (resp. $G'_e$) the graph obtained from $G$ (resp. $G'$) by playing $e$.
We will prove the two following points:
\begin{enumerate}[label=(\alph*)]
\item For every option $G_e$ of $G$, there is an option $G'_e$ of $G'$ such that $G_e \equiv G'_e$.
\item For every option $G'_e$ of $G'$, there is an option $G_e$ of $G$ such that $G_e \equiv G'_e$.
\end{enumerate}
By definition of the grundy value as the $\mex$ of the options' values, these two properties imply that~$G \equiv G'$.

First, consider the case where $e$ is an edge in both $G$ and $G'$. Then we have $G_e \equiv G'_e$. Indeed, suppose that we can apply a certain reduction operation on $G$, then the same reduction can be applied on $G_e$ since:
\begin{itemize}
\item a vertex that is useless in $G$ is also useless in $G_e$;
\item a vertex $v$ that is heavy in $G$ is also heavy in $G_e$, since any move that decrease the weight of $v$ also decrease the weight of one of its neighbors;
\item if two vertices are false twins in $G$, then they are also false twins in $G'$.
\end{itemize}
Hence, we can consider the graph obtained from $G_e$ by applying the same reduction operation as in $G$. Since $e$ is also an edge of $G'$, we can check easily that this graph is equal to $G'_e$. The total weight of $G_e$ is at most $k$, and by applying the induction hypothesis on $G_e$, we obtain $G_e \equiv G'_e$. 

Consequently, we only need to consider the options $G_e$ and $G'_{e'}$ where $e$ is not an edge of $G'$, and $e'$ is not an edge of $G$. For each of the three possible reduction rules, we will check that the two properties hold for these remaining options:

\begin{enumerate}[label=\textbf{Case \arabic*:},itemindent=*,leftmargin=0em, itemsep=1ex, listparindent=\parindent]
\item $G'$ is obtained from $G$ by deleting $v$, a useless vertex. 

Let $e$ be an edge of $G$, and consider the option $G_e$ of $G$. By the observation above, we can assume that $e$ is not an edge of $G'$, and consequently, we can write $e = (v,u)$ for some vertex $u \in N(v)$. Since $v$ is a useless vertex, we have $u \neq v$, and there is a loop on $u$. Let $e'$ be the edge $(u,u)$ in $G'$. Then $G'_{e'}$ is equal to the graph obtained from $G_e$ by removing the (useless) vertex $v$, and by applying the induction hypothesis on $G_e$, we have $G_e \equiv G'_{e'}$.

If $G'_{e'}$ is an option of $G'$, then $e'$ is also an edge of $G$, and this case is already handled by the observation above.

\item $G'$ is obtained from $G$ by deleting a heavy vertex $v$, and attaching loops to its neighbors.

Let $e$ be an edge of $G$, and consider the option $G_e$ of $G$ obtained by playing $e$. We can assume that $e$ is not an edge of $G'$, and consequently $e$ is incident to $v$ and we can write $e=(v,u)$ for some vertex $u \in N(v)$. Let $e'$ be the edge $(u,u)$ of $G'$. Then $G'_{e'}$ is the graph obtained from $G_e$ by simplifying the heavy vertex~$v$, and using the induction hypothesis on $G_e$, we have $G_e \equiv G'_{e'}$.

Consider now an option $G'_{e'}$ of $G'$ for some edge $e'$ of $G'$. We can assume that $e'$ is not an edge of $G$, and consequently $e' = (u,u)$ for a vertex $u$ adjacent to $v$ in $G$. Denote by $e$ the edge $(u,v)$ in $G$. Since playing $e'$ is a valid move in $G'$, we have that $\omega(u) > 0$, and since $v$ is a heavy vertex, $\omega(v) \geq \omega(u) > 0$. Consequently, playing $e$ on $G$ is also a valid move. Additionally, $v$ is still a heavy vertex of $G_e$, and $G'_{e'}$ is equal to the graph $G_e$ after the reduction of the heavy vertex $v$. Using the induction hypothesis on $G_e$, we have that $G'_{e'} \equiv G_{e}$.

\item $G'$ is obtained from $G$ by merging two false twins $v_1$ and $v_2$ into a single vertex $v$. 

First consider an option $G_e$ of $G$ for a certain edge $e$. We only need to consider the cases where $e$ is not in $G'$. If $e=(v_1, u)$, with $u$ a neighbor of $v_1$, consider the edge $e' = (v,u)$ of $G'$. Since playing $e$ on $G$ is a valid move, we have $\omega(u) > 0$ and $\omega(v_1) > 0$. By definition of the twin vertices reduction, we have $\omega(v) = \omega(v_1) + \omega(v_2) > 0$, hence playing $e'$ in $G'$ is a valid move. Additionally, we can check that $G'_{e'}$ is equal to the graph $G_e$ after merging the two twin vertices $v_1$ and $v_2$. As a consequence, using the induction hypothesis we have $G_e \equiv G'_{e'}$. If $e$ is the loop attached to $v_1$ in $G$, then by taking $e'$ the loop attached to $v$ in $G'$, and using a similar argument, we can show that $G_e \equiv G'_{e'}$.

Conversely, consider the option $G'_{e'}$ of $G'$ for some edge $e'$. We can assume that $e'$ is not an edge of $G$, and consequently $e = (v, u)$ for some vertex $u$. Since playing $e'$ on $G'$ is a valid move, one of $v_1$ or $v_2$ has a positive weight. Without loss of generality, assume $\omega(v_1)>0$. If $u \neq v$, then take $e = (v_1,u)$, otherwise, if $e'$ is the loop attached to $v$, take $e$ to be the loop attached to $v_1$. Then we can easily check that in both cases, playing $e$ on $G$ is a valid move. Additionally, $G'_{e'}$ is equal to the graph $G_e$ after reducing the two false twin vertices $v_1$ and $v_2$. By applying the induction hypothesis on $G_e$, we get that $G_e \equiv G'_{e'}$.
\end{enumerate}

Hence in all three cases, the properties holds, and consequently we have $G \equiv G'$. This ends the induction step and proves the proposition. 
\end{proof}

If a graph $G$ is not canonical, we can take a canonical reduction $G'$ of $G$. By Proposition~\ref{prop:reduction}, $G'$ has the same Grundy value as $G$, and we can study $G'$ instead. This allows to simplify the study of $G$ in many cases. In particular, Proposition~\ref{prop:reduction} gives a straightforward solution when $G$ is a star. Indeed, in this case, all the leaves are false twins and can be merged together without changing the Grundy value. The resulting graph only contains two adjacent vertices without loops.

Another simple consequence of Proposition~\ref{prop:reduction} is the following result:

\begin{corollary}
\label{cor:arckiswak}
Let $G=(V,E,\omega)$ be a weighted graph. There is a graph $G'$ such that the Grundy value of \lejeu\ on $G$ is the same as the Grundy value of \arck\ on $G'$. 
\end{corollary}
\begin{proof}
The graph $G'$ is constructed as follows: for every vertex $u$ with weight $\omega(u) > 1$, replace $u$ by $\omega(u)$ vertices, each with weight $1$, and each with the same neighbors as $u$. Then, for every vertex $u$ such that there is a loop attached to $u$, remove the loop and create a new vertex $u'$ with weight~$1$ adjacent to $u$. We can remark that $G'$ is obtained by applying on each vertex the inverse of the simplification procedure for false twins and heavy vertices. Indeed, each vertex $u$ is split into $\omega(u)$ false twins, and each vertex $u'$ created from a loop is heavy. Removing the heavy vertices, and merging back all the false twins by applying the simplification procedure gives back the graph $G$.

As a consequence, $G$ is a reduction of $G'$ obtained by merging the false twins and removing the heavy vertices that we created. Using Proposition~\ref{prop:reduction}, $G$ and $G'$ have the same Grundy values.
Since there is no loop in $G'$, and all the vertices of $G'$ have weight $1$, \lejeu\ played on $G'$ is just an instance of \arck.
\end{proof}
The reduction from \lejeu\ to \arck\ is not polynomial. Indeed, a vertex $v$ is transformed into a number $\omega(v)$ of vertices, which is exponential in the size of the binary representation of~$\omega(v)$.

\section{Trees of depth at most 2}
\label{sec:trees}


In this section, we will give a characterization of the outcome for \lejeu\ when the graph is a tree with depth at most~$2$. We begin by analyzing simple cases before moving on to more complicated ones.
Since no confusion is possible, the vertices will be named after their weights.

Given an unweighted graph $G$, the \emph{set of positions} for $G$ is the set of all possible weight functions $\omega$, denoted $\pos(G)$. Given an order $v_1, \ldots , v_n$ on the vertices of $G$, a specific weight function $\omega$ will be denoted by the tuple $(\omega(v_1), \ldots, \omega(v_n))$.

\begin{lemma}
\label{lem:oneVertexLoop}
We have $\grundy\big(\scalebox{0.7}{\soloop{a}}\big)=a\bmod2$.
\end{lemma}

\begin{proof}
The only available move is to play the loop. This decreases the number of counters on the vertex by 1, until it reaches 0. The result then holds by induction.
\end{proof}

\begin{lemma}
\label{lem:oneEdgeLoops}
The Grundy value of \scalebox{0.7}{\twoloops{a}{b}} is given by the formula:
$$\grundy\big( \scalebox{0.7}{\twoloops{a}{b}} \big)=((a+ b) \bmod 2) + 2 \times (\min(a,b) \bmod 2)$$
\end{lemma}

This is summarized in the table below where $m=min(a,b)$ and $M=max(a,b)$:

\begin{center}
\begin{tabular}{|c|c c|}
\hline
\backslashbox{m}{M} & even & odd \\ \hline
even & 0 & 1 \\
odd & 3 & 2 \\ \hline
\end{tabular}
\end{center}

\begin{proof}
If $a = 0$ or $b=0$, the vertex with weight zero can be removed from the graph. The resulting graph is just composed of a single loop, and the result follows from Lemma~\ref{lem:oneVertexLoop}. 

Let $a,b>0$. The Grundy values can be determined by induction as shown in Figure~\ref{table:twoloops}.

\begin{figure}[h]
\centering
\begin{tikzpicture}[scale=0.5]
\fill[gray!10!white] (-0.5, -0.5) rectangle (6,6) ;

\foreach \i/\j/\lab in {
0/0/0, 0/2/0, 0/4/0, 2/0/0, 4/0/0, 2/2/0, 
0/1/1, 1/0/1, 0/3/1, 0/5/1, 3/0/1, 5/0/1, 2/3/1, 3/2/1,
1/1/2, 1/3/2, 1/5/2, 3/1/2, 5/1/2, 3/3/2, 
1/2/3, 1/4/3, 2/1/3, 4/1/3
} {
\node at (\i, \j) {$\lab$} ;
}
\foreach \i in {
0, 1, 2, 3, 4, 5
} {
\node at (-1, \i) {$\i$} ;
\node at (\i, -1) {$\i$} ;
\draw[gray] (\i + 0.5, -0.5) -- (\i + 0.5, 6) ;
\draw[gray] (-0.5, \i + 0.5) -- (6, \i + 0.5) ;
}
\node[draw, inner sep = 0.5em, rounded corners = 1pt](a) at (2,3) { } ;
\draw[->] (a) -- (2,2.3) ;
\draw[->] (a) -- (1.3,3) ;
\draw[->] (a.south west) +(0.04,0.04) -- (1.3,2.3) ;

\node (a) at (-0.5,7) {$a$} ;
\node (b) at (7,-0.5) {$b$} ;

\draw[->] (-0.5,-0.5) -- (a) ;
\draw[->] (-0.5,-0.5) -- (b) ; 
\end{tikzpicture}

\caption[twoloops]{\label{table:twoloops} Table of grundy values for the graph \scalebox{0.7}{\twoloops{a}{b}} with different values for $a$ and $b$.}
\end{figure}
\end{proof}

If a graph is not connected, then by Observation~\ref{prop:playingOnANonConnectedGraph}, the Grundy value can be computed from the values of the connected components.
Since the graph in the following remark occurs several times in later proofs, we give an explicit characterization of its $\outcomeP$-positions. It is obtained from the two previous results.
\begin{remark}
\label{rk:sumOneLoppTwoLoops}
The graph \scalebox{0.7}{\soloop{a} $+$ \twoloops{b}{c}} is a $\outcomeP$-position if and only if one of the two following holds:
\begin{itemize}
\item $a, b$, and $c$ are even;
\item $a$ and $\max(b,c)$ are odd, and $\min(b,c)$ is even. 
\end{itemize}
\end{remark}

The two following proofs use the same argument. The idea is the following: for a fixed graph $G$, there are some ranges on the weights of the vertices for which the graph is not canonical, and by applying a reduction, the outcome of the graph can be computed by induction on the size of the graph. When the graph is canonical, we prove that a certain set $P \subseteq \pos(G)$ is the set of $\outcomeP$-position.
This argument is formalized in the following proposition. 
\begin{proposition}
\label{prop:genericProof}
Let $G$ be a graph, and $S \subseteq \pos(G)$ such that there is no move from a position not in $S$ to a position in $S$. Let $P$ be a subset of $S$, and assume that:
\begin{enumerate}[label=(\roman*)]
\item \label{it:noPtoP} There is no move from a position in $P$ to another position in $P$;
\item \label{it:PtonotSisN} From a position in $P$, any move to some position not in $S$ is a losing move.
\item \label{it:NtoP} From any position in $S \setminus P$, there is either a move to a position in $P$, or to a $\outcomeP$-position $s'$ not in $S$;
\end{enumerate}
Under these assumptions, a position $p$ in $S$ is a $\outcomeP$-position if and only if $p$ is in $P$.
\end{proposition}
\begin{proof}
The proof is by induction on the sum of the weights. Let $p$ be a position for $G$ such that $p \in S$.
If there is no possible move from $p$, then $p$ is a $\outcomeP$-position and $p \in P$, by condition~$\ref{it:NtoP}$. Thus we can suppose that $p$ has a nonempty set of options.

If $p \not \in P$, then by condition~$\ref{it:NtoP}$, there is a move from $p$ to a position $p'$ such that either $p' \not \in S$, and $p'$ is a $\outcomeP$-position, or $p'\in P$, and $p'$ is a $\outcomeP$-position by applying the induction hypothesis on $p'$. In both cases, $p'$ is a $\outcomeP$-position, and consequently $p$ is an $\outcomeN$-position.

Suppose now that $p \in P$, 
and let $p' \in opt(p)$. By condition~$\ref{it:noPtoP}$, we know that $p' \not \in P$. If $p' \in S$, then $p'$ is an $\outcomeN$-position by induction hypothesis. If $p' \not \in S$, then using condition~\ref{it:PtonotSisN}, $p'$ is also an $\outcomeN$-position. Consequently, $p$ is a $\outcomeP$-position.
\end{proof}


\begin{lemma}
\label{lemma:P4withloops}
Suppose that the graphs on the left is canonical, then the following outcome-equivalences holds:
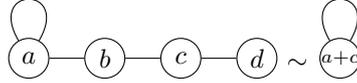
\begin{center}
\tikz[baseline=-4]{\node[vertex] (2) at (0,0) {$a$};\node[vertex] (1) at (1,0) {$b$};\node[vertex] (3) at (2,0) {$c$}; \node[vertex] (4) at (3,0) {$d$} ;\draw (4)--(3)--(1)--(2); \selfloop{2} ;} $\sim $ \soloop{\scriptstyle a+c}
\end{center}
\end{lemma}

\begin{proof}
Let $G$ be the graph on the left in the proposition.
Let $S \subset \pos(G)$ be the set of positions satisfying $c < b+d$, and $P \subset S$ be the subset of positions for which $a+c$ is even. Note that if a position $s$ is canonical , then $s \in S$. We want to show that $S$ and $P$ satisfy the three conditions of Proposition~\ref{prop:genericProof}.
Since any move decreases $a+c$ by exactly $1$, there is no move from a position in $P$ to another position in $P$, and point~\ref{it:noPtoP} holds.

If $s$ is a position in $P$ such that there is a move from $s$ to a position $s' \not \in S$, then necessarily $s'$ is obtained by playing the edge $(a,b)$. After the move, the vertex $c$ becomes a heavy vertex, and consequently $s'$ can be simplified to: 

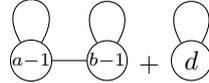
\begin{center} 	
\twoloops{\scriptstyle a-1}{\scriptstyle b-1} $+$ \soloop{d}
\end{center} 

Since $s \in P$ and $s' \not \in S$, we know that $c = b +d-1$, and $a + c$ is even. Consequently, $a-1 + b-1 + d = a + c-1$ is odd, which implies by Remark~\ref{rk:sumOneLoppTwoLoops} that $s'$ is an $\outcomeN$-position, and the point \ref{it:PtonotSisN} also holds.

Finally, let $s=(a,b,c,d)$ be a position such that $s \in S \setminus P$. Then one of $a$ or $c$ is odd, and thus non-zero. Consequently, there is a move from $s$ to a position $s'$ by playing either the loop attached to~$a$, or one of the edges $(b,c)$ or $(c, d)$. Since none of these move decrease the quantity $c - b - d$, we have $s' \in P$. So point \ref{it:NtoP} holds.

The three conditions of Proposition~\ref{prop:genericProof} are satisfied, and the result follows.
\end{proof}

The following Lemma is the key technical result that allows us to prove Theorem~\ref{thm:mainResult}. Using this result, the proof of the theorem will follow the following ideas. Given a tree $T$ of depth at most $2$, we can compute a more simple reduced graph, which has the same outcome as $T$. The reduced graph is a path on four vertices with a loop on one end. The outcome of the reduced graph can be computed either with the characterization of Lemma~\ref{lemma:P4withloops} if it is canonical, or by reducing the graph further to smaller components.

Note that in the statement of the Lemma, the graph on the left needs not be canonical. In particular, we may have $c \geq a + \sum_{i = 1}^k x_i$. The proof of the result is a bit technical but presents no theoretical difficulties. We simply check that the three conditions of Proposition~\ref{prop:genericProof} are satisfied, and proceed by case analysis.

\begin{lemma}
\label{th:mainTh}
Let $k>0$, and $x_1, \ldots, x_k$ and $y_1, \ldots, y_k$ be nonnegative integers such that $x_i > y_i$. Then the following holds
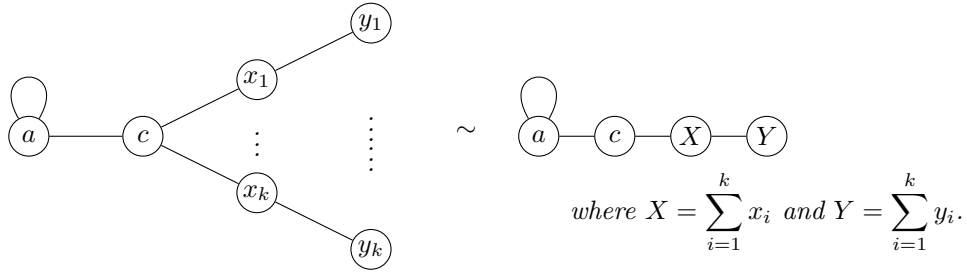
\begin{center}
\tikz[baseline=0]{
\node[vertex] (c) at (0,0) {$c$};
\node[vertex] (a) at (-1.5,0) {$a$};

\node[vertex] (x1) at (1.5,0.75) {$x_1$};
\node[vertex] (y1) at (3,1.5) {$y_1$};
\node[vertex] (xk) at (1.5,-0.75) {$x_k$};
\node[vertex] (yk) at (3,-1.5) {$y_k$};

\node (dots) at (1.5, 0) {$\vdots$};
\node (dots) at (3, 0.2) {$\vdots$};
\node (dots) at (3, -0.2) {$\vdots$};

\selfloop{a} ;
\draw (a) --(c) (c) -- (x1) -- (y1) (c) -- (xk) -- (yk);

}
$\qquad \sim \ $ 
\tikz[baseline=0]{
\node[vertex] (c) at (0,0) {$c$};
\node[vertex] (a) at (-1,0) {$a$};

\node[vertex] (X) at (1,0) {$X$};
\node[vertex] (Y) at (2,0) {$Y$};

\selfloop{a} ;
\draw (a) --(c) -- (X) -- (Y) ;	
\node at (2,-1) {where $X = \sum_{i = 1}^k x_i$ and $Y = \sum_{i = 1}^k y_i$.}
}

\end{center}	
\end{lemma}

\begin{proof}

Denote by $G_1, G_2, G_3, G_4$ the graphs shown on Figure~\ref{fig:th_moves}. On Figure~\ref{fig:th_moves}, the equivalence $G_2 \sim G_4$ if $X \geq Y+c$ is obtained by simplifying the heavy vertex $X$ in $G_2$. If $X < Y+c$, then either $c < a + X$, and in this case $G_2$ is canonical, and the outcome equivalence $G_2 \sim G_3$ is obtained by Lemma~\ref{lemma:P4withloops}, or $c \geq a +X$, and then the vertex $c$ is heavy. In this case, the outcome equivalence $G_2 \sim G_3$ is obtained by simplifying the vertex $c$, and then removing the now useless vertex~$Y$.

On the figure, the edges of these four graphs are marked with labels. Given an edge $e$ of $G_1$, we will denote by $f(e)$ the edge of $G_2$ with the same label as $e$. Additionally, if $s = (a, c, x_1, \ldots x_k, y_1, \ldots y_k)$ is a position of $G_1$, we will also denote by $f(s)$ the position $(a,c, X, Y)$ of $G_2$ where $X = \textstyle \sum_{i = 1}^k x_i$ and $Y = \textstyle \sum_{i=1}^k y_i$.
With these notations, we can easily check that for any two positions $s$ and $s'$ of $G_1$ such that there is a move from $s$ to $s'$, there is a move in $G_2$ from $f(s)$ to $f(s')$ by playing $f(e)$.

If either $X$ or $Y$ is equal to zero, then the outcome equivalence holds easily. Consequently, we will assume in the following $X > 0$ and $Y >0$.
We will prove the outcome-equivalence by induction on $k$. Clearly, the result holds when $k = 1$ since in this case both sides are the same. Consequently, we suppose $k > 1$, and assume that the results holds when there are less than $k$ branches.

Let $S$ be the set of positions $(a, c, x_1, \ldots x_k, y_1, \ldots y_k)$ of $G_1$ such that for all $i$, we have $x_i > y_i$. Let $P$ be the subset of positions $s \in S$ such that $f(s)$ is a $\outcomeP$-position of $G_2$. As previously, we will show that $S$ and $P$ satisfy the three conditions of Proposition~\ref{prop:genericProof}.

\begin{figure}[!ht]
\begin{tikzpicture}

\node[vertex] (c) at (0,0) {$c$};
\node[vertex] (a) at (-1.5,0) {$a$};

\node[vertex] (x1) at (1,0.75) {$x_1$};
\node[vertex] (y1) at (2,1.5) {$y_1$};
\node[vertex] (xk) at (1,-0.75) {$x_k$};
\node[vertex] (yk) at (2,-1.5) {$y_k$};

\node (dots) at (1, 0) {$\vdots$};
\node (dots) at (2, 0.2) {$\vdots$};
\node (dots) at (2, -0.2) {$\vdots$};

\draw[looseness = 8] (a) edge[in = 60, out=120] node[midway, above, inner sep=0pt] {$\scriptstyle (1)$} (a) ;
\draw (a) -- node[midway, above, inner sep=0pt] {$\scriptstyle (2)$} (c) 
(c) -- node[pos=0.6, above left, inner sep=-1pt] {$\scriptstyle (3)$} (x1) --  node[pos=0.6, above left, inner sep=-1pt] {$\scriptstyle (4)$} (y1) 
(c) --  node[pos=0.6, below left, inner sep=0pt] {$\scriptstyle (3)$} (xk) --  node[pos=0.6, below left, inner sep=0pt] {$\scriptstyle (4)$} (yk);

\node[fit=(a)(y1)(yk), draw, inner sep=10pt, rounded corners=10pt, color = gray] (box1) {} ;
\node[anchor= south west, inner sep=5pt, color = gray] (labbox1) at (box1.south west) {$G_1$} ;
\draw[gray, rounded corners = 3pt] (labbox1.north west) -| (labbox1.south east) ;

\node (eq1) at (3,0) {$\sim$} ;
\node[above = -0.1cm of eq1] {?} ;


\foreach \i/\lab in {0/$a$, 1/$c$, 2/$X$, 3/$Y$} {
\node[vertex] (p\i) at (3.8+1.2*\i, 0) {\lab} ;
}
\draw (p0) -- node[midway, above, inner sep=0pt] {$\scriptstyle (2)$} 
(p1) -- node[midway, above, inner sep=0pt] {$\scriptstyle (3)$}
(p2) -- node[midway, above, inner sep=0pt] {$\scriptstyle (4)$} (p3) ;
\draw[looseness = 8] (p0) edge[in = 60, out=120] node[midway, above, inner sep=0pt] (t) {$\scriptstyle (1)$} (p0) ;

\node[right= 5cm of box1.north, anchor = north, outer sep=0pt] (dum1) {} ;
\node[right= 5cm of box1.south, anchor = south, outer sep=0pt] (dum2) {} ;

\node[fit=(p0)(p3)(dum1)(dum2), draw, inner sep=0pt, rounded corners=10pt, color = gray, inner xsep = 5pt] (box2) {} ;
\node[anchor= south west, inner sep=5pt, color = gray] (labbox2) at (box2.south west) {$G_2$} ;
\draw[gray, rounded corners = 3pt] (labbox2.north west) -| (labbox2.south east) ; 

\node at (8.2,0) {$\sim$} ;

\node[vertex] (pa) at (9.7,0.8) {$a$} ;
\node[vertex] (pX) at (10.7,0.8) {$X$} ;
\draw[looseness = 8] (pa) edge[in = 60, out=120] node[midway, above, inner sep=0pt] (dum3) {$\scriptstyle (1)$} (pa) ; 
\draw[looseness = 8] (pX) edge[in = 60, out=120] node[midway, above, inner sep=0pt] {$\scriptstyle (4)$} (pX) ;

\node[vertex] (qa) at (9.7,-1.5) {$a$} ;
\node[vertex] (qc) at (10.9,-1.5) {$c$} ;
\node[vertex] (qY) at (11.9,-1.5) {$Y$} ;

\draw (qa) --  node[midway, above, inner sep=0pt] {$\scriptstyle (2)$}  (qc) ;
\draw[looseness = 8] (qa) edge[in = 60, out=120] node[midway, above, inner sep=0pt] (dum4) {$\scriptstyle (1)$} (qa) ; 
\draw[looseness = 8] (qc) edge[in = 60, out=120] node[midway, above, inner sep=0pt] {$\scriptstyle (3)$} (qc) ; 
\draw[looseness = 8] (qY) edge[in = 60, out=120] node[midway, above, inner sep=0pt] {$\scriptstyle (4)$} (qY) ; 

\draw [decorate,decoration={brace,amplitude=10pt}] (8.9,-2) -- (8.9,2) ;

\node[anchor=west] (if1) at (11.2, 1.3) {\scriptsize if $\scriptstyle X < Y +c$} ;

\node[anchor=west] (if2) at (12.2, -1) {\scriptsize if $\scriptstyle X \geq Y +c$} ;
\node[right= -13pt of if2, inner sep = 0pt] (dumif2) {} ;

\node[fit=(pa)(dum3)(if1), draw, inner sep=10pt, rounded corners=10pt, color = gray] (box3) {} ;
\node[anchor= south east, inner sep=5pt, color = gray] (labbox3) at (box3.south east) {$G_3$} ;
\draw[gray, rounded corners = 3pt] (labbox3.north east) -| (labbox3.south west) ; 

\node[fit=(qa)(dum4)(dumif2), draw, inner sep=10pt, rounded corners=10pt, color = gray] (box4) {} ;
\node[anchor= south east, inner sep=5pt, color = gray] (labbox4) at (box4.south east) {$G_4$} ;
\draw[gray, rounded corners = 3pt] (labbox4.north east) -| (labbox4.south west) ;

\end{tikzpicture}
\caption{\label{fig:th_moves} With $X = \sum_{i=1}^k x_i$ and $Y = \sum_{i=1}^k y_i$. The corresponding edges between the graphs are marked with the same labels. }
\end{figure}
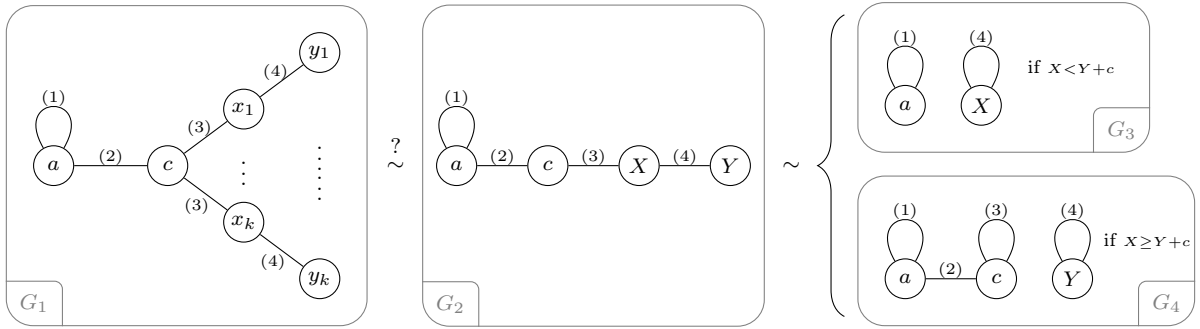

\begin{itemize}		
\item[$\ref{it:noPtoP}$] If $s$ and $s'$ are two positions in $P$, then there is no move from $s$ to $s'$. Indeed, suppose by contradiction that there is a move from $s$ to $s'$ by playing an edge $e$. Then $f(s)$ and $f(s')$ are both $\outcomeP$-positions of $G_2$ by definition of $P$, and there is a move from $f(s)$ to $f(s')$ by playing the edge $f(e)$, a contradiction.

\item[$\ref{it:PtonotSisN}$] Let $s = (a, c, x_1, \ldots x_k, y_1, \ldots y_k)$ be a position in $P$. Let us prove that there is no move from $s$ to a position $s' \not \in S$ such that $s'$ is a $\outcomeP$-position. Suppose by contradiction that this is not the case, and let $s'$ be a $\outcomeP$-position with $s' \not \in S$, such that there is a move from $s$ to $s'$. Since $s' \not \in S$, this move necessarily corresponds to playing an edge $e$ marked~$(3)$ in Figure~\ref{fig:th_moves}. Without loss of generality, we can suppose that $e = (c, x_k)$. Since before the move, we had $s \in S$, and after the move we have $s' \not \in S$, we know that $x_k = y_k + 1$. The position $s'$ is not canonical. Indeed, after playing the edge $e$, the vertex with weight $y_k$ becomes a heavy vertex. We can apply the simplification procedure which will first remove this vertex and put a loop on $x_k$, and then merge the vertices with weights $x_k -1$ and $a$. Consequently, $s'$ has the same outcome as:

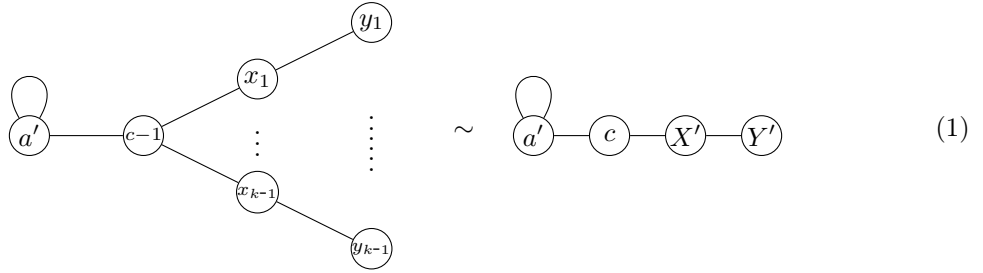
\begin{align}
\tikz[baseline=0]{			
\node[vertex] (c) at (0,0) {$\scriptstyle c-1$};
\node[vertex] (a) at (-1.5,0) {$a'$};
\node[vertex] (x1) at (1.5,0.75) {$x_1$};
\node[vertex] (y1) at (3,1.5) {$y_1$};
\node[vertex] (xk) at (1.5,-0.75) {$\scriptstyle x_{k\scalebox{0.85}{-}1}$};
\node[vertex] (yk) at (3,-1.5) {$\scriptstyle y_{k\scalebox{0.85}{-}1}$};
\node (dots) at (1.5, 0) {$\vdots$};
\node (dots) at (3, 0.2) {$\vdots$};
\node (dots) at (3, -0.2) {$\vdots$};
\selfloop{a} ;
\draw (a) --(c) (c) -- (x1) -- (y1) (c) -- (xk) -- (yk);
}
\qquad \sim \ 
\tikz[baseline=0]{
\node[vertex] (c) at (0,0) {$c$};
\node[vertex] (a) at (-1,0) {$a'$};
\node[vertex] (X) at (1,0) {$X'$};
\node[vertex] (Y) at (2,0) {$Y'$};
\selfloop{a} ;
\draw (a) --(c) -- (X) -- (Y) ;	
} \label{proof:reductionGraph}
\end{align}
where $a' = a + x_k -1$, $X' = X - x_k$ and $Y' = Y - Y_k$.	The equivalence above is obtained by applying induction hypothesis (there are only $k-1$ branches). Denote by $G'_2$ the simplified graph above.

If $X < c + Y$, then we have $G_2 \sim G_3$ (see Figure~\ref{fig:th_moves}). Since $s \in P$, we know that $G_2$ is a $\outcomeP$-position, and consequently $a + X$ is even. However, since $x_k = y_k + 1$, we also have $X - x_k\geq c -1 + Y - y_k$, and by the same argument, the outcome $G'_2$ is $\outcomeP$ if and only if $a' + X'$ is even. Since $a' + X' = a + x_k -1 + X - x_k = a + X -1$, this is not possible.

Thus, we can suppose that $X \geq c + Y$, and then $G_2 \sim G_4$ (see Figure~\ref{fig:th_moves}). Since $s \in P$, $G_2$ is a $\outcomeP$-position, and by Remark~\ref{rk:sumOneLoppTwoLoops} we know that $a + c + Y$ is even. By the equality $x_k = y_k + 1$, we also have $X - x_k \geq c-1 + Y - y_k$. Consequently, $G'_2$ is not canonical. Indeed, the vertex with weight $X'$ is a heavy vertex, and by applying the simplification procedure, $G'_2$ has the same outcome as:
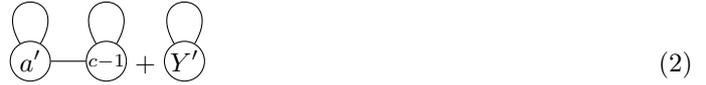
\begin{align}
\twoloops{a'}{\scriptstyle c-1} + \soloop{Y'} \label{proof:reductionGraph2}
\end{align}
Now, since $a +c + Y$ is even, we know that the quantity $a'+c-1 + Y' = a + x_k - 1+ c - 1 + Y - y_k = a + c - 1 + Y$ is odd, and by Remark~\ref{rk:sumOneLoppTwoLoops}, this implies that the position above is an $\outcomeN$-position. Consequently, $G'_2$ is also an $\outcomeN$-position, a contradiction.

\item[$\ref{it:NtoP}$] Finally, let $s = (a, c, x_1, \ldots x_k, y_1, \ldots y_k)$ be a position of $G_1$ such that $s \in S \setminus P$. We want to prove that either $s$ has a move to some position $s'$ in $P$, or there is a move to a $\outcomeP$-position not in~$S$. By assumption, $f(s)$ is a $\outcomeN$-position of $G_2$. We distinguish two possible cases:
\begin{itemize}[listparindent=\parindent]
\item There is an edge $e_2$ of $G_2$ such that playing $e_2$ from the position $f(s)$ is a winning move, and $e_2$ has a label different from $(3)$. Let $e \in f^{-1}(e_2)$ be an edge of $G_1$ with the same labels as $e_2$ such that playing $e$ from $s$ is a legal move, and let $s'$ be the position obtained from $s$ after playing $e$. Since $e$ has a label different from $(3)$, we have $s' \in S$, and consequently  $s' \in P$ since $f(s')$ is a $\outcomeP$-position. This proves that there is a move from $s$ to a position $s' \in P$.

\item The only winning move from $f(s)$ on $G_2$ is playing the edge with label $(3)$. As above take $e$ an edge of $G_1$ with label $(3)$ which is a possible move on $s$, and denote by $s'$ the position obtained from $s$ after playing $e$. If $s' \in S$, then $s' \in P$ since $f(s')$ is a $\outcomeP$-position. Thus, we can suppose $s' \not \in S$. We can also assume that playing an edge in $G_2$ with a label different than $(3)$ is a losing move. We want to prove that $s'$ is a $\outcomeP$-position.

If $X < Y +c$, then we can see from Figure~\ref{fig:th_moves} that there is a winning move in $G_2$ from the position $f(s)$ by playing an edge with label either $(1)$ or $(4)$. Thus, we can suppose that $X \geq Y + c$, which implies that the outcome of $G_2$ is the same as the outcome of $G_4$. We can also assume that the edge labeled $(3)$ is the only winning move for $G_4$. Indeed, if there was a winning move in $G_4$ with a label différent than $(3)$, there would be a winning move in $G_2$ with the same label.

Using the characterization of Grundy values from Lemma~\ref{lem:oneEdgeLoops}, we know that the grundy value of $G_4$ is $(a + c + Y \mod 2) + 2(\min(a,c) \mod 2)$. Since $G_4$ is an $\outcomeN$-position, this quantity is not zero. We can assume that $\min(a,c)$ is odd. Indeed, if it were even, then decreasing $Y$ by one by playing the edge $(4)$ would be a winning move on $G_4$. We also have $c < a$, since otherwise one of the two edges marked $(1)$ and $(2)$ would be a winning move. This implies that $c$ is odd. And finally, since playing the edge $(3)$ is a winning move, the position obtained by decreasing $c$ by one must be a $\outcomeP$-position, which implies that $a + c-1 + Y$ is even, and consequently, $a + Y$ is even.

Without loss of generality, we can assume that $e=(c,x_k)$, and $x_k = y_k + 1$. Using the same argument as in point $(ii)$, we know that $s'$ has the same outcome as the reduced graph in relation~(\ref{proof:reductionGraph}). Since we know that $X \geq Y + c$, and consequently $X - x_k \geq Y - y_k + c -1 $, this graph can be further simplified to the graph in~(\ref{proof:reductionGraph2}). We know that $c < a$, which implies that $a' = a + x_k - 1 > c-1$. Additionally, $c-1$ is even, and $Y' + a' = Y - y_k + a + x_k + 1 = Y +a$ is even. Hence $a$ and $Y$ have the same parity, and by Remark~\ref{rk:sumOneLoppTwoLoops} this implies that this graph is a $\outcomeP$-position, and consequently, $s'$ is a $\outcomeP$-position.
\end{itemize}
\end{itemize}

The three conditions of Proposition~\ref{prop:genericProof} hold, and the result follows.

\end{proof}

Theorem~\ref{thm:mainResult} is a corollary of this result. We recall that the depth of a vertex $u$ in a rooted tree is the number of edges in the path from the root to $u$ (in particular, the root has depth~0).

\begin{reptheorem}{thm:mainResult}
There is a polynomial time algorithm computing the outcome of \lejeu\ on any loopless tree of depth at most $2$.
\end{reptheorem}

\begin{proof}
Let $T$ be a rooted tree of depth at most $2$. In $T$, two leaves attached to the same vertex are false twins. Thus by Lemma~\ref{prop:reduction}, they can be merged without changing the outcome of the game. In the resulting tree, each vertex is adjacent to at most one leaf. Now, if a leaf at depth $2$ is heavy, then it can be removed, and a loop attached to its neighbor. The vertices with a loop are all false twins (they are all adjacent only to the root of the tree), and can then be merged into a single vertex. If a leaf is adjacent to the root, then we can attach a vertex with weight~0 to it. Let $T'$ be the resulting tree. 

By Proposition~\ref{prop:reduction}, we have $T\equiv T'$. By applying Lemma~\ref{th:mainTh}, the tree $T'$ can be further reduced to $P$, a path on four vertices with a loop on one end. If $P$ is not canonical, then it can again be reduced to one or several connected components, each with one or two vertices. The Grundy values of these components can be computed thanks to Lemma~\ref{lem:oneVertexLoop} and Lemma~\ref{lem:oneEdgeLoops}. Otherwise, if $P$ is canonical, and Lemma~\ref{lemma:P4withloops} gives the outcome value for $P$. Since all the reductions, and the characterization of Lemma~\ref{lemma:P4withloops} can be computed in polynomial time, this gives a polynomial time algorithm computing the outcome of any tree of depth at most $2$.
\end{proof}

We can see from Lemma~\ref{th:mainTh} and Lemma~\ref{lemma:P4withloops} that the outcome of \lejeu\ for a tree of depth at most 2 depends only on the parities of some of the weights and inequalities between them. This is not the case for $C_3$ (the cycle on three vertices), for which the periodicity of $\outcomeP$-positions does not follow this pattern but rather seems to depend on the values modulo 4.
A question that arises is whether more complex behaviours can emerge for larger or denser graphs.

However, if the weight of one vertex is large compared to the other weights, we will see in the next section that the behaviour remains simple.

\section{Periodicity}
\label{sec:perio}


In this section, we show a periodicity result on the outcome of \lejeu\ positions. More precisely, if we fix the number of counters for all vertices but one, say vertex $v_1$, the outcomes of this sequence of position is ultimately periodic, with period at most $2$. If there is no loop on $v_1$, when the weight of $v_1$ is large enough, it becomes a heavy vertex. Thanks to Proposition~\ref{prop:reduction} we already know that the sequence of outcomes is ultimately constant. The following result handles the case where there is a loop on vertex~$v_1$.

\begin{theorem}
\label{th:perio}
Let $G$ an unweighted graph with vertices $v_1, \ldots v_n$ such that there is a loop attached to $v_1$. Fix the integers $\omega_i \geq 0$ for $i \geq 2$, and let $\{S_x\}_{x\geq 0}$ be the sequence such that for every $x$, $S_x$ is the outcome of $(x, \omega_2, \ldots \omega_n) \in \pos(G)$. Then $\{S_x\}_{x\geq 0}$ is ultimately $2$-periodic with preperiod at most $2 \sum_{i \geq 2} \omega_i$.
\end{theorem}

\begin{proof}
We show this result by induction on $\Omega = \sum_{i \geq 2} \omega_i$. If the $\omega_i$ are all zeros, then $G(x, 0, \ldots 0)$ is equivalent to a graph with a single vertex and a loop attached to it. Its outcome is $\outcomeN$ if $x$ is odd, and $\outcomeP$ if $x$ is even.
Suppose that $\Omega > 0$. From the position $(x, \omega_2, \ldots \omega_n)$, there are three types of possible moves:
\begin{enumerate}
\item $(x-1, \omega_2, \ldots \omega_n)$ by playing on the loop attached to $v_1$.
\item $(x-1, \omega'_2, \ldots \omega'_n)$ with $\sum_{i \geq 2} \omega'_i = \Omega -1$ by choosing an edge adjacent to $v_1$.
\item $(x, \omega'_2, \ldots \omega'_n)$ with $\sum_{i \geq 2} \omega'_i = \Omega -2$ by choosing an edge not adjacent to $v_1$ (or $\Omega -1$ if the edge is a loop).
\end{enumerate}
Let $g$ be the function such that $g(x) = 1$ if, from position $(x, \omega_2, \ldots)$, there is a winning move of type $2.$ or $3.$, and $g(x) = 0$ otherwise. Using the induction hypothesis, $g(x)$ is ultimately periodic with period at most $2$, and preperiod at most $2 \Omega -1$.

Since the function $g$ takes values in $\{0,1\}$ and is $2$-periodic, there are only $4$ possibilities for the values of $g(2\Omega-1), g(2\Omega), \ldots$ . The possible values for $g(2 \Omega -1+ i)$ and the sequence of outcomes of the positions $(2 \Omega -1+i, \omega_2, \ldots))$, for $i \geq 0$ are summarized in Table~\ref{tab:periodicity}.
We can see that in all four cases in the table, the outcome is periodic starting at $i \geq 2\Omega$.

\begin{table}[!ht]
\centering
\def\arraystretch{2}
\begin{tabular}{|c|p{9cm}|}
\hline
$g(2\Omega-1 + i)$, $i \geq 0$		&	Sequence of outcomes of $(2\Omega-1+i, \omega_2, \ldots, \omega_n)$, $i \geq 0$ \\
\hline
$1,1, 1,1 \ldots$	& 	$\outcomeN, \outcomeN, \outcomeN \ldots$ since there is always a winning move of the form $2$ or $3$.\\
$0,0, 0, 0 \ldots$	& 	$\outcomeP, \outcomeN,\outcomeP, \outcomeN, \ldots$ or $\outcomeN, \outcomeP, \outcomeN, \outcomeP \ldots$, depending on the outcome of $(2\Omega-2, \omega_2, \ldots)$. Indeed, playing anything else than the loop attached to $v_1$ is a losing move and the outcome alternates between $\outcomeN$ and $\outcomeP$. \\
$1,0, 1,0 \ldots$	&	$\outcomeN, \outcomeP, \outcomeN, \outcomeP \ldots$. Indeed, if $i$ is even, then $g(2\Omega-1 + i) =1$, consequently there is a winning move of type $2$ or $3$. If $i$ is odd, then moves of type $2$ and $3$ are losing move, and so is the move on the loop attached to $v_1$.\\
$0,1, 0,1 \ldots$	& 	$X,\outcomeN, \outcomeP,\outcomeN, \outcomeP,\ldots$ where $X$ can be either $\outcomeN$ or $\outcomeP$. Indeed, if $i$ is odd, then $g(2\Omega-1 + i) =1$, consequently there is a winning move of type $2$ or $3$. If $i$ is even and different from $0$, then moves of type $2$ and $3$ are losing move, and so is the move on the loop attached to $v_1$. When $i=0$, the outcome can be either $\outcomeP$ or $\outcomeN$ depending on the outcome of $(2\Omega-2, \omega_2, \ldots)$. \\
\hline
\end{tabular}
\caption{\label{tab:periodicity} Table of periodicity of the sequence of outcomes of $(2\Omega-1+i, \omega_2, \ldots, \omega_n)$, $i \geq 0$ depending on the periodic values of $g$.}
\end{table}

\end{proof}

\begin{corollary}
Given an unweighted graph $G$, the sequence $\{\grundy((x, \omega_2 \ldots, \omega_n))\}_{x \geq 0}$ of Grundy values for positions of $G$ is ultimately $2$-periodic. If the Grundy values in the periodic part are bounded by $k$, then there is a constant $c_k$ only depending on $k$ such that the preperiod is at most $2\sum_{i \geq 2} \omega_i + c_k$.
\end{corollary}
\begin{proof}
Simply observe that we can replace $G$ by $G+U$ (the disjoint union of $G$ and $U$) in the previous statement, for any other graph $U$. Now if the Grundy value of $(\omega'_1, \ldots \omega'_p) \in \pos(U)$ is $k$, then $(x, \omega_2, \ldots, \omega'_1\ldots) \in \pos(G+U)$ is a $\outcomeP$-position if and only if $(x, \omega_2, \ldots) \in \pos(G)$ has Grundy value $k$. Consequently, by applying the result of Theorem~\ref{th:perio} on $G+U$, we know that the Grundy values of $(x, \omega_2, \ldots)\in \pos(G)$ are ultimately $2$-periodic. The preperiod is at most $2 \sum \omega_i + 2 \sum \omega'_i$. Taking $U$ and $(\omega'_i)_{i \geq1}$ such that $\sum \omega'_i$ is the smallest possible with $\grundy((\omega'_1, \ldots)) = k$, and noting $c_k = 2 \sum \omega'_i$ gives the desired result.
\end{proof}

\section{Unboundedness of Grundy values}
\label{sec:unbound}

The problem of finding a graph family with unbounded Grundy values is open for a large number of vertex and edge deletion games. Recently, some results have been achieved for the game \textsc{Graph Chomp}~\cite{recentChomp}. However, for most of these games, the graph families that are studied tend to have ultimately periodic Grundy sequences. For example, in their study of \nodek\ in~\cite{fleischer}, the authors found increasingly many irregularities in the non-periodic parts of the Grundy sequences for subdivided stars with three paths, but gave no indication as to whether the irregular values were bounded or not.

In this section, we prove that the Grundy values of \lejeu\ are unbounded. This, coupled with the fact that any position of \lejeu\ has an equivalent position of \arck, as was shown in Section~\ref{sec:first}, also proves that the Grundy values of \arck\ are unbounded. This result holds even if we restrict ourselves to only play on forests. This answers a problem posed in~\cite{huggan}, where the unboundedness of the Grundy values for \arck\ 
was conjectured. 
Since \arck\ played on a graph $G$ is \nodek\ played on the line graph of $G$, this also implies that the Grundy values of \nodek\ are unbounded.

We remind the reader that the grundy value of a non-connected graph can be obtained from the values of its connected components:

\begin{repobs}{prop:playingOnANonConnectedGraph}
Let $G$ be a non-connected graph, and $G_1,\ldots,G_k$ its connected components.
We have: $$\grundy(G)=\grundy(\sum_{i=1}^{k} G_i)=\bigoplus_{i=1}^{k}\grundy(G_i).$$
\end{repobs}

We prove the following result:

\begin{reptheorem}{prop:grundyValuesUnbounded}
The Grundy values for the game \lejeu~are unbounded.
\end{reptheorem}

\begin{proof}
We inductively build a sequence of graphs, $G_1, G_2, G_3\ldots$ such that for any $i \neq j$, $G_i$ and $G_j$ have different Grundy values. We construct $G_i$ in such a way that there is a vertex $u_i$ of $G_i$ with a loop attached on $u_i$ such that playing the loop is a winning move.

We take $G_1 = \scalebox{0.7}{\soloop{1}}$. There is only one move on $G_1$ and it is a winning move.

Given a positive integer $n$, suppose that we have built the graphs $G_1,\ldots,G_{n}$ with this property. For $1 \leq i \leq n$, let us denote by $u_i$ the vertex of $G_i$ such that there is a winning move on $G_i$ by playing the loop attached to $u_i$. We construct the graph $G_{n+1}$ in the following way, pictured on Figure~\ref{fig:grundyValuesUnbounded}:
\begin{enumerate}
\item For all $i \leq n$, we create two copies $G'_i$ and $G''_i$ of $G_i$;
\item We create a vertex $u_{n+1}$ of weight 1 with a loop, which is connected to the vertex $u'_i$ of every $G'_i$.
\end{enumerate}

\begin{figure}[!ht]
\centering
\begin{tikzpicture}
\node[vertex] (un) at (0,0) {$\scriptscriptstyle u_{n+1}$};
\draw[looseness=8] (un) edge[in=60, out=120] (un);

\node[vertex] (un-1) at (-2,-2) {$u'_{1}$};
\draw[looseness=8] (un-1) edge[in=60, out=120] (un-1);
\draw (un)--(un-1);
\draw (un-1)--(-2.75,-3.5)--(-1.25,-3.5)--(un-1);
\draw (-2,-3.15) node {$G'_{1}$};
\draw (0,-2) node {\ldots};
\node[vertex] (u1) at (2,-2) {$u'_{n}$};
\draw[looseness=8] (u1) edge[in=60, out=120] (u1);
\draw (un)--(u1);
\draw (u1)--(2.75,-3.5)--(1.25,-3.5)--(u1);
\draw (2,-3.15) node {$G'_{n}$};

\node[vertex] (un-1b) at (4,-1) {$u''_{1}$};
\draw[looseness=8] (un-1b) edge[in=60, out=120] (un-1b);
\draw (un-1b)--(4.75,-2.5)--(3.25,-2.5)--(un-1b);
\draw (4,-2.15) node {$G''_{1}$};
\draw (5.5,-1) node {\ldots};
\node[vertex] (u1b) at (7,-1) {$u''_{n}$};
\draw[looseness=8] (u1b) edge[in=60, out=120] (u1b);
\draw (u1b)--(7.75,-2.5)--(6.25,-2.5)--(u1b);
\draw (7,-2.15) node {$G''_{n}$};
\end{tikzpicture}
\caption{The inductive construction of the graph $G_{n+1}$. Note that every vertex has a weight of 1.}
\label{fig:grundyValuesUnbounded}
\end{figure}
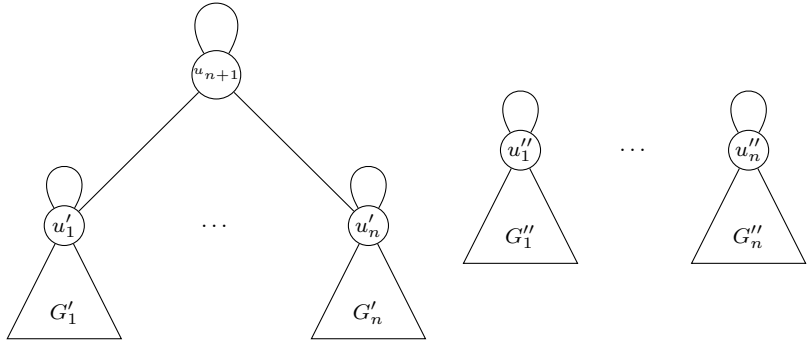

Let us now prove that \ref{it:premier} for all $i\leq n$, we have $\grundy(G_{n+1}) \neq \grundy(G_i)$ and \ref{it:second} playing on the loop attached to $u_{n+1}$ leads to a $\outcomeP$-position:
\begin{enumerate}[label=(\roman*)]
\item \label{it:premier} Let $i \leq n$, we will show that there is an option $G'$ of $G_{n+1}$ such that $\grundy(G') = \grundy(G_{i})$. Using the definition of the Grundy value with the $mex$ operator, this implies that $\grundy(G_{n+1}) \neq \grundy(G_i)$. Let $G'$ be the option of $G_{n+1}$ obtained by playing on the edge $(u'_i,u_{n+1})$. Since a vertex with no counters left can be removed without changing the game, we now delete both vertices and their incident edges. The graph $G'$ is composed of several components, and can be written as the disjoint sum:
$\sum_{\substack{j=1 \\ j \neq i}}^{n} ( G_j + G_j ) + G'_i + G_i$ where $G'_i$ is the graph obtained from $G_i$ after playing on the loop attached to $u_i$. Using the induction hypothesis, we know that $G'_i$ is a $\outcomeP$-position, hence $\grundy(G'_i) = 0$. Proposition~\ref{prop:playingOnANonConnectedGraph} ensures that the Grundy value of $G'$ is:

$$\grundy(G') =  \bigoplus_{\substack{j=1 \\ j \neq i}}^{n} ( \grundy(G_j) \oplus \grundy(G_j) ) \oplus 0 \oplus \grundy(G_i) = \grundy(G_i)$$
\item \label{it:second} Let $G'$ be the the graph obtained from $G_{n+1}$ by playing on the loop attached to $u_{n+1}$. Then $G'$ can be written as the disjoint sum $G' = \sum_{j=1}^{n} (G_j+G_j)$, which has Grundy value~0 and thus is a $\outcomeP$-position.
\end{enumerate}
Thus, the Grundy value of $G_{n+1}$ is different from the Grundy values of the $G_i$ for any $i \leq n$, and a winning move is to play on the loop attached to $u_{n+1}$. This completes the induction step.
\end{proof}

As shown with Corollary~\ref{cor:arckiswak}, from any position of \lejeu, one can compute a position of \arck\ with the same Grundy value. Moreover, any position of \arck\ can be changed into an equivalent position of \nodek. Thus Theorem~\ref{prop:grundyValuesUnbounded} implies the following:

\begin{corollary}
\label{cor:grundyValuesUnboundedArcK}
The Grundy values for the games \arck~and \nodek~are unbounded.
\end{corollary}

The construction in the proof of Theorem~\ref{prop:grundyValuesUnbounded} gives a family of graphs of exponential size (by induction, $G_n$ has $3^{n-1}$ vertices). Since all the vertices have weight 1, the \arck\ positions that we obtain by applying the construction described in the proof of Corollary~\ref{cor:arckiswak} are of similar size.
It may be of interest to find a family of graphs with unbounded Grundy values and of polynomial size, both for \lejeu\ and for \arck.

%
%
%

\end{document}